\documentclass[11pt]{article} 

\usepackage{amsmath}
\usepackage{amssymb}
\usepackage{color}
\usepackage{array}
\usepackage[T1]{fontenc}
\usepackage[latin1]{inputenc}
\usepackage{datetime}
\usepackage[linkcolor=black,urlcolor=black,colorlinks=true]{hyperref}

\def\r{\rightarrow}

\usepackage{enumerate}

\newcommand{\fdem}{\hspace*{\fill}~$\Box$\par\endtrivlist\unskip}

\renewcommand{\L}{\mathbb{L}}

\newcommand{\N}{\mathbb{N}}

\newcommand{\C}{\mathbb{C}}

\renewcommand{\r}{\mathop{\rightarrow}}

\newcommand{\cB}{\mbox{$\cal B$}}

\newtheorem{theo}{Theorem}
\newtheorem{pro}{Proposition}
\newenvironment{proof}[1]{\textit{Proof#1.\,}}{\fdem}
\newtheorem{lem}{Lemma}
\newtheorem{rem}{Remark}
\newtheorem{ex}{Example}
\newtheorem{cor}{Corollary}

\title{Computable bounds of $\ell^2$-spectral gap for discrete Markov chains with band transition matrices}

\author{Loïc HERV\'E, and James LEDOUX \footnote{INSA de Rennes, IRMAR, F-35042, France; CNRS, UMR 6625, Rennes, F-35708, France; Université Européenne de Bretagne, France. \{Loic.Herve,James.Ledoux\}@insa-rennes.fr}
}
\date{}
\begin{document}

\maketitle
%


\begin{abstract}
We analyse the $\ell^2(\pi)$-convergence rate of irreducible and aperiodic Markov chains with $N$-band transition probability matrix $P$ and with invariant distribution $\pi$. This analysis is heavily based on: first the study of the essential spectral radius $r_{ess}(P_{|\ell^2(\pi)})$ of $P_{|\ell^2(\pi)}$ derived from Hennion's quasi-compactness criteria; second  the connection between  the Spectral Gap property (SG$_2$) of $P$ on $\ell^2(\pi)$ and the $V$-geometric ergodicity of $P$. Specifically, (SG$_2$)  is shown to hold under the condition   
\begin{equation*} 
\alpha_0 := \sum_{{m}=-N}^N \limsup_{i\r+\infty} \sqrt{P(i,i+{m})\, P^*(i+{m},i)}\ <\, 1. 
\end{equation*}  
Moreover $r_{ess}(P_{|\ell^2(\pi)}) \leq \alpha_0$. Effective bounds on the convergence rate can be provided from a truncation procedure. 
\end{abstract}
\begin{center}

\end{center}
\section{Introduction}
Let $P:=(P(i,j))_{(i,j)\in\N^2}$ be a Markov kernel on the countable state space $\N$.  
Throughout the paper we assume that $P$ is irreducible and aperiodic, that $P$ has a unique invariant probability measure denoted by $\pi:=(\pi(i))_{i\in\N}$, and finally that
\begin{equation} \label{ass-voisin}
\exists i_0\in\N,\ \exists N\in\N^*,\ \forall i\geq i_0\ :\quad |i-j| > N\ \Longrightarrow\ P(i,j)=0. \tag{\textbf{AS1}}
\end{equation}
We denote by $(\ell^2(\pi),\|\cdot\|_2)$ the  Hilbert space of sequences $(f(i))_{i\in\N}\in\C^{\N}$ such that $\|f\|_2 := [\, \sum_{i\geq0} |f(i)|^2\, \pi(i)\, ]^{1/2} < \infty$. Then $P$ 
defines a linear contraction on $\ell^2(\pi)$, and its adjoint operator $P^*$ on $\ell^2(\pi)$ is defined by $P^*(i,j) := \pi(j)\, P(j,i)/\pi(i)$. If $\pi(f):=\sum_{i\geq0} f(i)\, \pi(i)$, then the kernel $P$ is said to have the spectral gap property on  $\ell^2(\pi)$ if there exists $\rho\in(0,1)$ and $C\in(0,+\infty)$ such that 
\begin{equation} \label{ineg-gap-gene}
\forall n\geq1, \forall f\in\ell^2(\pi),\quad \|P^nf - \Pi f\|_2 \leq C\, \rho^n\, \|f\|_2 \quad \text{with} \quad \Pi f := \pi(f) 1_{\N}. \tag{\textbf{SG$_2$}}
\end{equation}
 A standard issue is to compute the value (or to find an upper bound) of 
\begin{equation} \label{def-varho-gene}
\varrho_2 := \inf\{\rho\in(0,1) : \text{(\ref{ineg-gap-gene}) holds true}\}.
\end{equation}
In this work the quasi-compactness criteria of \cite{Hen93} is used to study (\ref{ineg-gap-gene}) and to estimate $\varrho_2$. In Section~\ref{sec-bounded-tp} it is proved that (\ref{ineg-gap-gene}) holds when 
\begin{equation} \label{ass-beta}
\alpha_0 := \sum_{{m}=-N}^N \limsup_{i\r+\infty} \sqrt{P(i,i+{m})\, P^*(i+{m},i)}\ <\, 1. \tag{\textbf{AS2}}
\end{equation}
Moreover $r_{ess}(P_{|\ell^2(\pi)}) \leq \alpha_0$. 
We refer to \cite{Hen93} for the definition of the essential spectral radius $r_{ess}(T)$ and for quasi-compactness of a bounded linear operator $T$ on a Banach space.  Under the  assumptions 
\begin{gather} 
 \quad \forall {m}=-N,\ldots,N,\quad P(i,i+{m}) \xrightarrow[i\r +\infty]{} a_{m}\in[0,1] \label{cond-lim-intro} \tag{\textbf{AS3}}\\
\frac{\pi(i+1)}{\pi(i)} \xrightarrow[i\r +\infty]{} \tau \in[0,1) \label{pi-tail} \tag{\textbf{AS4}}\\
\sum_{k=-N}^{N} k\, a_{k}\, < 0, \label{MoySautBorne} \tag{\textbf{NERI}}
\end{gather}
Property~(\ref{ass-beta}) holds (hence (\ref{ineg-gap-gene})) and $\alpha_0$ can be explicitly computed in function of $\tau$ and the $a_{m}$'s. 
Moreover, using the inequality  $r_{ess}(P_{|\ell^2(\pi)}) \leq \alpha_0$, Property~(\ref{ineg-gap-gene}) is proved to be connected to the  $V$-geometric ergodicity of $P$ for $V:= (\pi(n)^{-1/2})_{n\in\N}$. 
In particular, denoting the minimal $V$-geometrical ergodic rate by $\varrho_V$, it is proved that, either $\varrho_2$ and $\varrho_V$ are both less than $\alpha_0$, or $\varrho_2=\varrho_V$. As a result, an accurate bound of $\varrho_2$ can be obtained for random walks (RW) with i.d.~bounded increments using the results of \cite{HerLedJAP14}. Actually, any estimation of $\varrho_V$, for instance that derived in Section~\ref{sec-tronc} from the truncation procedure of \cite{HerLed14a}, provides an estimation of $\varrho_2$. We point out that all the previous results hold without any reversibility properties. 

The spectral gap property for Markov processes has been widely investigated in the discrete and continuous-time cases (e.g.~see \cite{Ros71,Che04}). There exist different definitions of the spectral gap property according that we are concerned with discrete or continuous-time case (e.g. see \cite{Yue00,MaoSon13}). The focus of our paper is on the discrete time case. In the reversible case, the equivalence between the geometrical ergodicity and (\ref{ineg-gap-gene}) is proved in \cite{RobRos97} and  Inequality $\varrho_2 \leq \varrho_{V}$ is obtained in \cite[Th.6.1.]{Bax05}.  This equivalence fails in the non-reversible case (see \cite{KonMey12}). The link between $\varrho_2$ and $\varrho_{V}$ stated in our Proposition~\ref{pro-RW-SG} is obtained with no reversibility condition. Formulae for $\varrho_2$ are provided in \cite{StaWub11,Wu12} in terms of isoperimetric constants which are related to $P$ in reversible case and to $P$ and $P^*$ in non-reversible case. However, to the best of our knowledge, no explicit value (or upper bounds) of $\varrho_2$ can be derived from these formulae for discrete Markov chains with band transition matrices. Our explicit bound $r_{ess}(P_{|\ell^2(\pi)}) \leq \alpha_0$ in Theorem~\ref{theo-spec-gap-gene} is the preliminary key results in this work. Recall that $r_{ess}(P_{|\ell^2(\pi)})$ is a natural lower bound of $\varrho_2$ (apply \cite[Prop.~2.1]{HerLedJAP14} with the Banach space $\ell^2(\pi)$).  The essential spectral radius of Markov operators on a $\L^2$-type space is investigated for Markov chains with general state  space in \cite{Wu04}, but no explicit bound for $r_{ess}(P_{|\ell^2(\pi)})$ can be derived a priori from these theoretical results for  Markov chains with band transition matrices, except in the reversible case \cite[Th.~5.5.]{Wu04}. 

\section{Property~(\ref{ineg-gap-gene}) and $V$-geometrical ergodicity} \label{sec-bounded-tp}
%
\begin{theo} \label{theo-spec-gap-gene}
If Condition~\emph{(\ref{ass-beta})} holds, then $P$ satisfies \emph{(\ref{ineg-gap-gene})}. Moreover $r_{ess}(P_{|\ell^2(\pi)}) \leq \alpha_0$. 
\end{theo}
%
\begin{proof}{} 
Consider the Banach space $\ell^1(\pi):=\{(f(i))_{i\in\N}\in\C^{\N}$: $\|f\|_1 := \sum_{i\geq0} |f(i)|\, \pi(i) < \infty\}$. 
%
%
\begin{lem} \label{lem-D-F-gene}
For any $\alpha>\alpha_0$, there exists a positive constant $L\equiv L(\alpha)$ such that 
$$\forall f\in\ell^2(\pi),\quad \|Pf\|_2 \leq \alpha\, \|f\|_2 + L\|f\|_1.$$
\end{lem}
%
Since the identity map is compact from $\ell^2(\pi)$ into $\ell^1(\pi)$ (from the Cantor diagonal procedure), it follows from Lemma~\ref{lem-D-F-gene} and from \cite{Hen93} that $P$ is quasi-compact on $\ell^2(\pi)$ with $r_{ess}(P_{|\ell^2(\pi)}) \leq \alpha$. Since $\alpha$ can be chosen arbitrarily close to $\alpha_0$, this gives $r_{ess}(P_{|\ell^2(\pi)}) \leq \alpha_0$. Then (\ref{ineg-gap-gene}) is deduced from aperiodicity and irreducibility assumptions.
\end{proof}
\begin{proof}{}
Under Assumption~(\ref{ass-voisin}), define 
\begin{equation} \label{beta-k}
\forall i\geq i_0,\ \forall {m}=-N,\ldots,N,\quad \beta_{m}(i) := \sqrt{P(i,i+{m})\, P^*(i+{m},i)}.
\end{equation} 
Let $\alpha>\alpha_0$, with $\alpha_0$ given in (\ref{ass-beta}). Fix $\ell\equiv \ell(\alpha) \geq i_0$ such that 
$\sum_{{m}=-N}^{N} \,  \sup_{i\geq \ell} \beta_{m}(i) \leq \alpha$. For $f\in\ell^2(\pi)$ we have from Minkowski's inequality and the band structure of $P$ for $i\geq \ell$ 
\begin{eqnarray}
\|Pf\|_2 &\leq& \bigg[\sum_{i<\ell}\big|(Pf)(i)\big|^2\pi(i)\, \bigg]^{1/2} 
+\ \bigg[\sum_{i\geq \ell} \bigg|\sum_{{m}=-N}^{N} P(i,i+{m})\, f(i+{m})\bigg|^2\pi(i)\, \bigg]^{1/2} \nonumber \\ 
&\leq& L\, \sum_{i<\ell} |(Pf)(i)|\, \pi(i)\,  + \bigg[\sum_{i\geq \ell} \bigg|\sum_{{m}=-N}^{N} P(i,i+{m})\, f(i+{m})\bigg|^2\pi(i)\, \bigg]^{1/2} \label{ineg-DF-inter1}
\end{eqnarray}
where $L\equiv L_\ell>0$ comes from equivalence of norms on $\C^{\ell}$. Moreover we have $\sum_{i<\ell}|(Pf)(i)|\, \pi(i) \leq \|Pf\|_1 \leq \|f\|_1$. To control the second term in (\ref{ineg-DF-inter1}), define $F_{m}=(F_{m}(i))_{i\in\N}\in\ell^2(\pi)$ by $F_{m}(i):= P(i,i+{m})\, f(i+{m}) (1 - 1_{\{0,\ldots,\ell-1\}}(i))$ for $-N \leq m \leq N$. Then 
$$
\bigg[\sum_{i\geq \ell} \bigg|\sum_{{m}=-N}^{N} P(i,i+{m})\, f(i+{m})\bigg|^2\pi(i)\, \bigg]^{1/2}= \big\|\sum_{{m}=-N}^{N} F_{m}\, \big\|_2 
 \leq    \sum_{{m}=-N}^{N} \| F_{m}\|_2.
$$
\begin{eqnarray*}
\text{and }\quad \| F_{m}\|_2^2 &=& \sum_{i\geq \ell} P(i,i+{m})^2\, |f(i+{m})|^2\pi(i) \\[-2mm]
&=& \sum_{i\geq \ell} P(i,i+{m})\frac{\pi(i)\, P(i,i+{m})}{\pi(i+{m)}}\, |f(i+{m})|^2\pi(i+{m)} \\[-1mm] 
& \leq &  \sup_{i\geq \ell} \beta_{m}(i)^2\, \|f\|_2^2 \qquad \text{(from the definition of $P^*$ and from (\ref{beta-k}))}.\\[-3mm]
\end{eqnarray*}
The statement in Lemma~\ref{lem-D-F-gene} can be deduced from the previous inequality and from (\ref{ineg-DF-inter1}). 
\end{proof}

The core of our approach to estimate $\varrho_2$ is the relationship between Property~(\ref{ineg-gap-gene}) and the $V-$geometric ergodicity. Indeed, specify Theorem~\ref{theo-spec-gap-gene} in terms of the $V-$geometric ergodicity with $V := ({\pi(n)}^{-1/2})_{n\in\N}$. Let $(\cB_{V},\|\cdot\|_{V})$ denote the space of sequences $(g(n))_{n\in\N}\in\C^\N$ such that $\|g\|_{V} := \sup_{n\in\N} V(n)^{-1}\, |g(n)| < \infty$.
Recall that $P$ is said to be $V$-geometrically ergodic if $P$ satisfies the spectral gap property on $\cB_V$, namely: there exist $C\in(0,+\infty)$ and $\rho\in (0,1)$ such that  
\begin{equation} \label{ineg-gap-V} 
\forall n\geq1, \forall f\in\cB_V,\quad \|P^nf - \Pi f\|_V \leq C\, \rho^n\, \|f\|_V. \tag{\textbf{SG$_V$}} 
\end{equation}
When this property holds, we define 
\begin{equation} \label{def-rho-V} 
\varrho_V := \inf\{\rho\in(0,1) : \text{(\ref{ineg-gap-V}) holds true}\}.
\end{equation} 

\begin{rem} \label{rem_fle-alpha}
Under Assumptions~\emph{(\ref{cond-lim-intro})} and \emph{(\ref{pi-tail})}, we have 
\begin{equation} \label{fle-alpha}
\alpha_0 := \sum_{{m}=-N}^N \limsup_{i\r+\infty} \sqrt{P(i,i+{m})\, P^*(i+{m},i)} 
 = \begin{cases}
 \text{$\displaystyle\sum_{{m}=-N}^{N} a_{m}\, \tau^{-{m/2}}$} & \text{if }\  \tau\in(0,1) \\[0.12cm]
\quad  a_0 & \text{if }\  \tau=0.  \\
\end{cases}
\end{equation}
Indeed, if \emph{(\ref{pi-tail})} holds with $\tau\in(0,1)$, then the claimed formula follows from the definition of $P^*$. If $\tau=0$ in \emph{(\ref{pi-tail})}, then $a_{m}=0$ for every ${m}=1,\ldots,N$ from $\sum_{m=-N}^N P(i+m,i)\, \pi(i+m)/\pi(i) = 1$. Thus $a_{-m}
=0$ when ${m}<0$. Hence $\alpha_0=a_0$.   
\end{rem}
\begin{pro}  \label{pro-RW-SG} 
If $P$ and $\pi$ satisfy Assumptions~\emph{(\ref{cond-lim-intro})}, \emph{(\ref{pi-tail})} and \emph{(\ref{MoySautBorne})}, then $P$ satisfies~\emph{(\ref{ass-beta})} (with $\alpha_0<1$ given in (\ref{fle-alpha})). Moreover $P$ satisfies both \emph{(\ref{ineg-gap-gene})} and \emph{(\ref{ineg-gap-V})} with $V := ({\pi(n)}^{-1/2})_{n\in\N}$, we have $\max(r_{ess}(P_{|{\cal B}_V}),r_{ess}(P_{|\ell^2(\pi)})) \leq \alpha_0$, and the next  assertions hold: 
\begin{enumerate}[(a)]
  \item if $\varrho_{V} \leq \alpha_0$, 
	then $\varrho_2 \leq \alpha_0$; 
	\item if $\varrho_{V} > \alpha_0$, 
	then $\varrho_2 = \varrho_{V}$.
\end{enumerate}
\end{pro}
\begin{proof}{}
If $\tau=0$ in (\ref{pi-tail}), then $\alpha_0=a_0<1$ from (\ref{fle-alpha}) and (\ref{MoySautBorne}). Now assume that (\ref{pi-tail}) holds with $\tau\in(0,1)$. Then $\alpha_0=\sum_{m=-N}^{N} a_{m}\, \tau^{-{m/2}} = \psi(\sqrt{\tau})$, where:  
$\forall t>0,\ \psi(t) := \sum_{k=-N}^{N} a_{k}\, t^{-k}$. Moreover it easily follows from the invariance of $\pi$ that $\psi(\tau)=1$. Inequality $\alpha_0 =\psi(\sqrt{\tau}) < 1$ is deduced from the following assertions: $\forall t \in(\tau,1),\ \psi(t)<1$ and $\forall t \in (0,\tau) \cup (1,+\infty),\ \psi(t) > 1$. To prove these properties, note that $\psi(\tau)=\psi(1)=1$ and that $\psi$ is convex on $(0,+\infty)$. Moreover we have $\lim_{t\r+\infty} \psi(t)= +\infty$ since $a_k>0$ for some $k<0$ (use $\psi(\tau)=\psi(1)=1$ and $\tau\in(0,1)$). Similarly, $\lim_{t\r 0^{+}} \psi(t) =  +\infty$ since $a_k>0$ for some $k>0$. This gives the desired properties on $\psi$ since $\psi'(1) > 0$ from~(\ref{MoySautBorne}).

(\ref{ineg-gap-gene}) and $r_{ess}(P_{|\ell^2(\pi)}) \leq \alpha_0$ follow from Theorem~\ref{theo-spec-gap-gene}. Next (\ref{ineg-gap-V}) is deduced from the well-known link between geometric ergodicity and the following drift inequality: 
\begin{equation} \label{drift-V}
\forall\alpha\in(\alpha_0,1),\ \exists L\equiv L_\alpha>0, \quad PV \leq \alpha V + L\, 1_{\N}. 
\end{equation}
This inequality holds from $\lim_i (PV)(i)/V(i) = \alpha_0$. 

Then (\ref{ineg-gap-V}) is derived from (\ref{drift-V}) using aperiodicity and irreducibility. It also follows from  (\ref{drift-V}) that $r_{ess}(P_{|{\cal B}_V}) \leq \alpha$ (see \cite[Prop.~3.1]{HerLedJAP14}). Thus $r_{ess}(P_{|{\cal B}_V}) \leq \alpha_0$. 

Now we prove $(a)$ and $(b)$ using the spectral properties of \cite[Prop.~2.1]{HerLedJAP14} of both $P_{|\ell^2(\pi)}$ and $P_{|{\cal B}_V}$ (due to quasi-compactness, see \cite{Hen93}). We will also use the following obvious inclusion: $\ell^2(\pi) \subset \cB_{V}$. In particular  every eigenvalue of $P_{|\ell^2(\pi)}$ is also an eigenvalue for $P_{|{\cal B}_V}$. 
First assume that $\varrho_{V} \leq \alpha_0$. Then there is no eigenvalue for $P_{|{\cal B}_V}$ in the annulus $\Gamma:=\{\lambda\in\C : \alpha_0 < |\lambda| < 1\}$ since $r_{ess}(P_{|{\cal B}_V}) \leq \alpha_0$. From $\ell^2(\pi) \subset \cB_{V}$ it follows that there is also no eigenvalue for $P_{|\ell^2(\pi)}$ in this  annulus. Hence $\varrho_2 \leq \alpha_0$ since $r_{ess}(P_{|\ell^2(\pi)}) \leq \alpha_0$. 
Second assume that $\varrho_{V} > \alpha_0$. Then $P_{|{\cal B}_V}$ admits an eigenvalue $\lambda\in\C$ such that $|\lambda| = \varrho_{V}$. Let $f\in\cB_{V}$, $f\neq 0$, such that $Pf = \lambda f$. We know from \cite[Prop.~2.2]{HerLedJAP14} that there exists some $\beta\equiv\beta_\lambda\in(0,1)$ such that $|f(n)| = \text{O}(V(n)^{\beta}) = \text{O}(\pi(n)^{-\beta/2})$, so that $|f(n)|^2\pi(n) = \text{O}(\pi(n)^{(1-\beta)})$, thus $f\in \ell^2(\pi)$ from (\ref{pi-tail}). We have proved that $\varrho_2 \geq \varrho_{V}$. Finally the converse inequality is true since every eigenvalue of $P_{|\ell^2(\pi)}$ is an eigenvalue for $P_{|{\cal B}_V}$. Thus $\varrho_2 = \varrho_{V}$.  
\end{proof}

From Proposition~\ref{pro-RW-SG}, any estimation of $\varrho_V$ provides an estimation of $\varrho_2$. This is illustrated in Example~\ref{ex-rw-gene} and Corollary~\ref{th-first}. Markov chains in Example~\ref{ex-rw-gene} have been studied in details in \cite[Section~3]{HerLedJAP14}. Also mention that further technical details are reported in \cite{HerLedHal15}. 
\begin{ex} [RWs with i.d.~bounded increments] \label{ex-rw-gene}
Let $P$ be defined as follows. There exist some positive integers $c,g,d\in\N^*$ such that 
\begin{subequations}
\begin{gather*}
\forall i\in\{0,\ldots,g-1\},\quad \sum_{j = 0}^{c} P(i,j)=1; \label{cond-bound-prob} \\[-1mm]
\forall i\ge g, \forall j\in\N, \quad P(i,j) = \begin{cases}
 a_{j-i} & \text{if }\  i-g\leq j \leq i+d \\
 0 & \text{otherwise.}  \\
\end{cases} \label{Def_HRW-ter} \\[-1mm]
(a_{-g},\ldots,a_d)\in[0,1]^{g+d+1} : a_{-g}>0, \ a_d>0, \ \sum_{k=-g}^{d} a_k=1. 
\label{Def_HRW-bis} 
\end{gather*} 
\end{subequations}
%


\noindent Assume that $P$ is aperiodic and irreducible, and satisfies~\emph{(\ref{MoySautBorne})}.
Then $P$ has a unique invariant distribution $\pi$. It can be derived from standard results of linear difference equation that $\pi(n) \sim c\, \tau^n$ when $n\r+\infty$, with $\tau\in(0,1)$ defined by $\psi(\tau)  = 1$, where $\psi(t) := \sum_{k=-N}^{N} a_{k}\, t^{-k}$. Thus, if $\gamma := \tau^{-1/2}$, then $\cB_V=\{(g(n))_{n\in\N}\in\C^\N,\ \sup_{n\in\N} \gamma^{-n}\, |g(n)| < \infty\}$. Then we know from \cite[Prop.~3.2]{HerLedJAP14} that $r_{ess}(P_{|{\cal B}_V}) =\alpha_0$ with $\alpha_0$ given in (\ref{fle-alpha}), and that $\varrho_{V}$ can be computed from an algebraic polynomial elimination. From this computation, Proposition~\ref{pro-RW-SG} provides an accurate estimation of $\varrho_2$. Property~\emph{(\ref{ineg-gap-gene})} was proved in \cite[Th.~2]{Wu12} under an extra weak reversibility assumption (with no explicit bound on $\varrho_2$). However, except in case $g=d=1$ where reversibility is automatic, an RW with i.d.~bounded increments is not reversible or even weak reversible in general. No reversibility condition is required here. 
\end{ex}
\section{Bound for $\varrho_2$ via truncation} \label{sec-tronc}
Let $P$ be any Markov kernel on $\N$, and let us consider the $k$-th truncated (and augmented on the last column) matrix $P_k$ associated with $P$ as in \cite{HerLed14a}. 
If $\sigma(P_k)$ denotes the set of eigenvalues of $ P_k$, define 
$\rho_k := \max\big\{|\lambda|,\, \lambda\in\sigma(P_k),\, |\lambda| <1\big\}$. The weak perturbation method in \cite{HerLed14a} provides the following general result where Condition~(\ref{ass-voisin}) is not required and $V$ is any unbounded increasing sequence. 
\begin{pro} \label{pro-gene-tronc}
Let $P$ be an irreducible and aperiodic Markov kernel on $\N$ satisfying the following drift inequality for some unbounded increasing sequence $(V(n))_{n\in\N}$:
\begin{equation} \label{drift-V-bis}
\exists\delta\in[0,1[,\ \exists L>0, \quad PV \leq \delta V + L\, 1_{\N}. 
\end{equation}
Let $\varrho_V$ be defined in (\ref{def-rho-V}). Then, either $\varrho_V \leq \delta$ and $\limsup_k\rho_k \leq \delta$, or $\varrho_V > \delta$ and $\varrho_{V} = \lim_k\rho_k$. 
\end{pro}
\begin{proof}{} 
Condition~(\ref{drift-V-bis}) ensures that the assumptions of \cite[Lem.~6.1]{HerLed14a} are satisfied, so that $r_{ess}(P_{|{\cal B}_V})\leq\delta$. Then, using standard duality arguments, the spectral rank-stability property \cite[Lem.~7.2]{HerLed14a} applies to $P_{|{\cal B}_{V}}$ and $P_k$. If $\varrho_V \leq \delta$, then, for each $r$ such that $\delta<r<1$, $\lambda = 1$ is the unique eigenvalue of $P_{|{\cal B}_{V}}$ in $C_r := \{\lambda\in\C : r < |\lambda| \leq 1\}$ (see \cite{Hen93}). From \cite[Lem.~7.2]{HerLed14a}  this property holds for $P_k$ when $k$ is large enough, so that $\limsup_k\rho_k \leq r$. Thus $\limsup_k\rho_k \leq \delta$ since $r$ is arbitrarily close to $\delta$. Now assume that $\varrho_V > \delta$, and let $r$ be such that $\delta < r < \varrho_V$. Then $P_{|{\cal B}_{V}}$ has a finite number of eigenvalues in $C_r$, say $\lambda_0,\lambda_1,\ldots,\lambda_N$, with $\lambda_0=1$, $|\lambda_1| = \varrho_V$ and $|\lambda_k| \leq \varrho_V$ for $k=2,\ldots,N$ (see \cite{Hen93}). 
For $a\in\C$ and $\varepsilon>0$ we define $D(a,\varepsilon) := \{z\in\C : |z-a| < \varepsilon\}$. Now consider any $\varepsilon>0$ such that the disks $D(\lambda_k,\varepsilon)$ for $k=0,\ldots,N$ are disjoint and are contained in $C_r$ pour $k\geq1$. From \cite[Lem.~7.2]{HerLed14a}, for $k$ large enough, $1$ is the only eigenvalue of $P_k$ in $D(1,\varepsilon)$, the others eigenvalues of $P_k$ in $C_r$ are contained in $\cup_{k=1}^ND(\lambda_k,\varepsilon)$, and finally each $D(\lambda_k,\varepsilon)$ contains at least one eigenvalue of $P_k$. Thus each eigenvalue $\lambda\neq 1$ of $P_k$ in $C_r$ has modulus less than $\varrho_V + \varepsilon$, so that $\rho_k \leq \varrho_V + \varepsilon$. Moreover the disk $D(\lambda_1,\varepsilon)$ contains at least an eigenvalue $\lambda$ of $P_k$, so that $\rho_k \geq |\lambda| \geq \varrho_V - \varepsilon$. Thus, for $k$ large enough, we have $\varrho_V - \varepsilon\leq \rho_k \leq \varrho_V + \varepsilon$. 
\end{proof}

Under the assumptions of Proposition~\ref{pro-RW-SG} we deduce the following result from Proposition~\ref{pro-gene-tronc}. 
%
\begin{cor} \label{th-first}
If $P$ satisfies the assumptions of Proposition~\ref{pro-RW-SG}, then the following properties holds with $\alpha_0$ given in (\ref{fle-alpha}): 
\begin{enumerate}[(a)]
	\item $\varrho_2 \leq \alpha_0\ \Longleftrightarrow \varrho_{V} \leq \alpha_0$, and in this case we have $\limsup_k\rho_k \leq \alpha_0$; 
	\item $\varrho_2 > \alpha_0\ \Longleftrightarrow \varrho_{V} > \alpha_0$, and in this case we have $\varrho_2 = \varrho_{V} = \lim_k\rho_k$. 
\end{enumerate}
\end{cor}
As usual the reversible case is simpler. In particular  we can take $C=1$ and $\rho=\varrho_2$ in (\ref{ineg-gap-gene}). Details and numerical illustrations for Metropolis-Hastings kernels are reported in \cite{HerLedHal15}. 
 
\bibliographystyle{alpha}

\end{document}